\documentclass[11pt,twoside,reqno,a4paper]{amsart}
\usepackage{mathptmx}
\usepackage{amsmath}     
\usepackage{amssymb}
\usepackage{array}
\usepackage{geometry}
\usepackage[bookmarks=true,colorlinks=false, pdfstartview=FitV, linkcolor=black, citecolor=blue, urlcolor=black]{hyperref}

\usepackage{color}
\definecolor{DarkRed}{rgb}{0.55,.00,0.2}
\definecolor{DarkGrey}{rgb}{0.35,.35,0.35}

 \newtheorem{theorem}{Theorem}[section]
\newtheorem{lemma}[theorem]{Lemma}

\theoremstyle{definition}

\theoremstyle{remark}

\numberwithin{equation}{section}

\newcommand{\ds}{\displaystyle}
\newcommand{\e}{{\rm e}}

\title
{ On a nonorthogonal polynomial sequence associated with Bessel operator}

\begin{document}

\author{Ana F. Loureiro}
\address{{\normalfont{\bf Ana F. Loureiro} and {\bf S. Yakubovich}.} CMUP \& DM-FCUP, Rua do Campo Alegre,  687; 4169-007 Porto (Portugal)}
\curraddr{}
\thanks{Work of AFL supported by Funda\c c\~ao para a Ci\^encia e Tecnologia via the grant SFRH/BPD/63114/2009. The research of AFL and SY was in part supported by the ``Centro de Matem\'atica'' of the
University of Porto.}

\author{P. Maroni}
\address{{\normalfont\bf P. Maroni}, LJLL -- CNRS, Universit\'e Pierre et Marie Curie, B\^oite courrier 187, 75252 Paris cedex 05, France}
\curraddr{}
\email{}
\thanks{}

\author{S. Yakubovich}
\curraddr{}
\email{(AFL,SY) \{anafsl, syakubov\}@fc.up.pt; (PM) maroni@ann.jussieu.fr}


\date{}

\dedicatory{}

\commby{}

\begin{abstract}
By means of the Bessel operator a polynomial sequence is constructed to which several properties are given. Among them, its explicit expression, the connection with the Euler numbers, its integral representation via the Kontorovich-Lebedev transform. Despite its non-orthogonality, it is possible to associate to the canonical element of its dual sequence a positive-definite measure as long as certain stronger constraints are imposed. 
\end{abstract}

\maketitle

\section{\normalfont Introduction and preliminaries}
The modified Bessel function of third kind (also known as the {\it Macdonald function}, specially in russian literature) $K_{i\tau}(x)$ of the argument $x>0$ and the pure imaginary subscript $i\tau$ is an eigenfunction of the (Bessel) operator 
\begin{equation} \label{Bessel Operator}
		\mathcal{A} = x^2 - x \frac{d}{dx} x \frac{d}{dx} \ = \  - x^2  \frac{d^2}{dx^2} - x \frac{d}{dx} + x^2
\end{equation}
for the associated eigenvalues $\tau^2$, {\it i.e.}, 
\begin{equation} \label{A Kitau}
	\mathcal{A} K_{i\tau}(x) = \tau^{2} K_{i\tau}(x)  
\end{equation}
and naturally providing the identity on the integral powers of the Bessel operator $\mathcal{A}$
\begin{equation} \label{An Kitau}
	\mathcal{A}^n K_{i\tau}(x) = \tau^{2n} K_{i\tau}(x) \ , \ n\in\mathbb{N}, 
\end{equation}
inductively defined. 
Throughout this paper, $\mathbb{N}$ will denote the set of all positive integers, $\mathbb{N}_{0}=\mathbb{N}\cup \{0\}$, whereas $\mathbb{R}$ and $\mathbb{C}$ will denote, respectively, the field of the real and complex numbers. By $\mathbb{R}^+$ and $\mathbb{R}_{0}^+$ we respectively mean the set of all positive and nonnegative real numbers. Further notations are introduced as needed throughout the text.

Defined also by the cosine Fourier transform 
\begin{equation*}\label{Cosine Fourier K}
	K_{i\tau }(x) = \int_{0}^\infty \e^{-x \cosh(u)}\cos(\tau u)du \ , \ x\in\mathbb{R}^+, \ \tau\in\mathbb{R}
\end{equation*}
the function $K_{i\tau}$ is real valued and represents a kernel of the following operator of the Kontorovich-Lebedev transformation \cite{YakubovichBook1996}, given by the formula
\begin{equation}\label{KL oper }
	\mathcal{K}_{i\tau}[f] = \int_{0}^\infty  K_{i\tau}(x) f(x) dx \ , 
\end{equation} 
is an isometric isomorphism (see \cite{Yakubovich2004JAT}) between the two Hilbert spaces $L_{2}(\mathbb{R}^+; x dx)$ and  $L_{2}(\mathbb{R}^+;\tau \sinh\pi \tau d\tau)$. Moreover the inverse formula  holds 
\begin{align} \label{KL oper Reciprocal}
	& f(x) =  \frac{2}{\pi^{2} x}\int_{0}^{\infty}  \tau \sinh \pi\tau \  \ K_{i\tau}(x) \mathcal{K}_{i\tau}[f]  \ dy 
\end{align}
where the convergence of the latter integral is understood with respect to the norm in $L_{2}(\mathbb{R}^+; x dx)$.

The Kontorovich-Lebedev transform has been used in many applications including,
for instance, fluid mechanics, quantum and nano-optics and plasmonics. This
transform, where integration occurs over the index of the function rather than over
the argument, proves to be useful in solving the resulting differential equations when
modeling optical or electronic response of such problems.

Recently  \cite{Yakubovich2009}, the action of any power of the operator $\mathcal{A}$ over $\e^{-x}$ was investigated. Precisely, the sequence 
$\ 
	\{  \e^{-x} \mathcal{A}^{n} \e^{x}\}_{n\geqslant 0}
\ $
was shown to be a polynomial sequence (PS), spanning the vector space $\mathcal{P}$ -- the set of all polynomials with coefficients in $\mathbb{C}$ -- since each of its elements has exactly degree $n$. However, this sequence is orthogonal with respect to the singular measure $\delta$ (the Dirac measure). 
In the present work,  we propose instead to investigate the functions 
$$
	  \e^{-x}x^{-\alpha} \mathcal{A}^{n} \e^{x}x^\alpha\ , \quad n\in \mathbb{N}_{0},
$$
which end up to be a polynomials of exactly degree $n$ as long as $\Re(\alpha)>-1/2$, as it will be revealed in Section \ref{sec: polynomial seq pn}. In this case the sequence $\{ p_{n}(\cdot;\alpha)\}_{n\geqslant 0}$ defined by 
\begin{equation}\label{def pn nonmonic}
	p_{n}(x;\alpha)=(-1)^{n}{\rm e}^{x}\ x^{-\alpha} \mathcal{A}^{n}\  {\rm e}^{-x} x^{\alpha}
	\ , \quad n\in \mathbb{N}_{0}, 
\end{equation}
is indeed a PS. 
The properties of $\{ p_{n}(\cdot;\alpha)\}_{n\geqslant 0}$ will be thoroughly revealed Section \ref{sec: polynomial seq pn}, where a generating function, the connection with the generalized Euler polynomials or the Bernoulli polynomials (when the parameter $\alpha$ ranges in $\mathbb{N}$) will be expounded. On the grounds of these developments lies the Kontorovich-Lebedev transform. 
The differential-difference equations over the polynomials $p_{n}(x;\alpha)$ are crucial for the subsequent developments related to the corresponding dual sequence of monic polynomial sequence (MPS) 
${\{ P_{n}(x;\alpha):=a_{n}^{-1} p_{n}(x;\alpha)\}_{n\geqslant 0}}$ where $a_{n}\neq0$ are the leading coefficients of the polynomials $p_{n}(x;\alpha)$, $n\in\mathbb{N}_{0}$.

The research about the MPS $\{ P_{n}(x;\alpha)\}_{n\geqslant 0}$ will then proceed toward the analysis of the corresponding dual sequence, $\{ u_{n}(\alpha)\}_{n\geqslant 0}$, whose elements, called as {\it forms} (or {\it linear functionals}), belong to the dual space $\mathcal{P}'$ of $\mathcal{P}$ and are defined according to 
$$
	\langle u_{n}(\alpha),P_{k}(\cdot;\alpha) \rangle := \delta_{n,k}, \; n,k\geqslant 0,
$$ 
where $\delta_{n,k}$ represents the {\it Kronecker delta} function. Here, by $\langle u,f\rangle$ we mean the action of $u\in\mathcal{P}'$ over $f\in\mathcal{P}$, but a special notation is given to the action over the elements  of the canonical sequence $\{x^{n}\}_{n\geqslant 0}$ -- the {\it moments of $u\in\mathcal{P}'$}:   $(u)_{n}:=\langle u,x^{n}\rangle, n\geqslant 0 $. Any element $u$ of $\mathcal{P}'$ can be written in a series of any dual sequence $\{ \mathbf{v}_{n}\}_{n\geqslant 0}$ of a MPS  $\{B_{n}\}_{n\geqslant 0}$ \cite{MaroniTheorieAlg}: 
\begin{equation} \label{u in terms of un}
	u = \sum_{n\geqslant 0} \langle u , B_{n} \rangle \; \mathbf{v}_{n} \; .
\end{equation}
To infer differential equations for the elements of the dual sequence it is important to recall that a linear operator $T:\mathcal{P}\rightarrow\mathcal{P}$  
has a transpose $^{t}T:\mathcal{P}'\rightarrow\mathcal{P}'$ defined
by 
\begin{equation}\label{Ttranspose}
\langle{}^{t}T(u),f\rangle=\langle u,T(f)\rangle\,,\quad u\in\mathcal{P}',\: f\in\mathcal{P}.
\end{equation}
For example, for any form $u$ and any polynomial $g$, let $ Du=u'$ and $gu$ be the forms defined as usual by
$\langle u',f\rangle :=-\langle u , f' \rangle \ ,\  \langle gu,f\rangle :=\langle u, gf\rangle ,
$  
where $D$ is the differential operator \cite{MaroniTheorieAlg}. Thus, $D$ on forms is minus the transpose of the differential operator $D$ on polynomials.

In  Section \ref{sec: Dual sequence} we will come to the conclusion that it is possible to associate to $u_{0}(\alpha)$ a definite-positive measure as long as $\Re(\alpha)>0$. This implies the regularity of any form proportional to $u_{0}(\alpha)$, a concept that is recalled in detail on page \pageref{regular form}. Concomitantly, the existence of a unique MPS regularly orthogonal (hereafter MOPS) with respect to $u_{0}(\alpha)$  is ensured, that is, there exists a unique $\{Q_{n}(x;\alpha)\}_{n\geqslant 0}$ such that 
$\langle u_{0}(\alpha) , Q_{n}(x;\alpha)Q_{m}(x;\alpha)\rangle = k_{n}(\alpha) \delta_{n,m}$ with $k_{n}\neq0$ for any $n,m\in\mathbb{N}_{0}$.

\section{\normalfont Algebraic and differential properties
}\label{sec: polynomial seq pn}

The functions defined by \eqref{def pn nonmonic} are actually polynomials of exactly degree $n$, as previously claimed and below proved, and therefore $\{ p_{n}(x;\alpha)\}_{n\geqslant 0}$ forms a PS. 
Once this is guaranteed we will proceed to obtain an explicit expression for these polynomials $ p_{n}(x;\alpha)$ in \S \ref{subsec: the cn nu}, an integral representation for such polynomials by means of the Kontorovich-Lebedev transform in \S  \ref{subsec GF} and finally, by inverting this integral transform, to derive a relation between the generalized Euler polynomials and $p_{n}(x;\alpha)$ in \S \ref{subsec Euler gen}.

\begin{lemma}\label{lem: equations for pn} For any complex number $\alpha$ such that $\Re(\alpha)>-1/2$ the functions $p_{n}(x;\alpha)$ defined in \eqref{def pn nonmonic} are actually polynomials of exactly degree $n$ fulfilling  
\begin{equation}
		\label{dif difference eq over small pn}
	p_{n+1}(x;\alpha)
	= x^{2} p_{n}''(x;\alpha) - x(2x-1-2\alpha) p_{n}'(x;\alpha)
				- \Big((2\alpha+1)x-\alpha^{2}\Big) p_{n}(x;\alpha) \ , \ n\in\mathbb{N}_{0},
\end{equation}
and also
\begin{equation}\label{pn x alpha +1}
	(2\alpha+1) \ x \ p_{n}(x;\alpha+1) 
				= - p_{n+1}(x;\alpha)  + \alpha^{2} p_{n}(x;\alpha)  \ , \ n\in\mathbb{N}_{0}.
\end{equation} 
Moreover, $p_{n}(0;\alpha)=\alpha^{2n}$ and $p_{n}(x;\alpha)=(-2)^{n} (\alpha +1/2)_{n} x^{n} +a_{n-1}(x)$ with $\deg a_{n-1}\leqslant n-1$ for any $n\in\mathbb{N}_{0}$ (under the convention $a_{-1}(x)=0$) and $(y)_{n}$ representing the Pochhammer symbol: $(y)_{n}=\prod\limits_{\tau=0}^{n-1}(y+\tau)$ for $n\in\mathbb{N}$ and $(y)_{0}=1$.  
\end{lemma}

\begin{proof}
According to the definition \eqref{def pn nonmonic}, we may write 
$$
	 p_{n+k}(x;\alpha) 
	= (-1)^k {\rm e}^{x} x^{-\alpha} \mathcal{A}^k \big( {\rm e}^{-x} x^{\alpha} \big({\rm e}^{x} x^{-\alpha}
	 \mathcal{A}^n {\rm e}^{-x} x^{\alpha} \big)\big)
	\ , \ n,k\in\mathbb{N} , 
$$
providing 
\begin{equation}\label{Ak pn}
		 p_{n+k}(x;\alpha) 
	=(-1)^k {\rm e}^{x} x^{-\alpha} \mathcal{A}^k \big( {\rm e}^{-x} x^{\alpha} p_{n}(x;\alpha) \big)
	\ , \ n,k\in\mathbb{N} .
\end{equation}
On behalf of the property, 
\begin{equation}\label{A over ex xalpha f}
	\e^{x}x^{-\alpha} \mathcal{A}(\e^{-x}x^{\alpha} f(x)) 
	= -x^{2} \frac{d^{2}}{dx^{2}} f(x) + x(2x-1-2\alpha) \frac{d}{dx}f(x) 
	+ \Big((2\alpha+1)x-\alpha^{2}\Big) f(x),
\end{equation}
that holds for any analytic function $f$, 
the particular choice of $k=1$ in \eqref{Ak pn} furnishes \eqref{dif difference eq over small pn}, whereas the choice of $n=1$ with $k$ varying within $\mathbb{N}$ leads to the equality \eqref{pn x alpha +1}.  

Insofar as computing \eqref{def pn nonmonic} for $n=0$ and $n=1$ we respectively extract that 
$p_{0}(x;\alpha)=1$ and $p_{1}(x;\alpha)=-(2\alpha+1)x+\alpha^{2}  $, then 
on the basis of \eqref{dif difference eq over small pn}, by finite induction we derive that $p_{n}(x;\alpha)$ is a polynomial that has exactly degree $n$ as long as $\Re(\alpha)>-1/2$. 

Finally, the particular choice of $x=0$ in \eqref{pn x alpha +1} provides $p_{n+1}(0;\alpha)=\alpha^{2n+2} p_{0}(0;\alpha)$, while from \eqref{dif difference eq over small pn} we deduce that the leading coefficient $c_{n,n}{(\alpha)}$ of  $p_{n}(x;\alpha)$ satisfies $c_{n+1,n+1}{(\alpha)}=-(2n+2\alpha+1)c_{n,n}{(\alpha)}$ for any $n\in\mathbb{N}$, whence the result. 
\end{proof}

Directly from \eqref{Ak pn}, we obtain 
$$
	\mathcal{A}^k \big( {\rm e}^{-x} x^{\alpha} p_{n}(x;\alpha) \big)
	= \mathcal{A}^n \big( {\rm e}^{-x} x^{\alpha} p_{k}(x;\alpha) \big)
	\ , \ n,k\in\mathbb{N} . 
$$

From the insertion of \eqref{pn x alpha +1} into \eqref{dif difference eq over small pn} we derive 
\begin{equation}\label{pn x alpha +1 Diff Rel}
		(2\alpha +1) p_{n}(x;\alpha+1) 
		= -x\ p_{n}''(x;\alpha) + (2x-1-2\alpha) p_{n}'(x;\alpha) + (1+2\alpha) p_{n}(x;\alpha)
		, \  n\geqslant 0. 
\end{equation}

The polynomial sequence  treated in  \cite{Yakubovich2009} corresponds to the special case $\{p_{n}(x;0)\}_{n\geqslant 0}$. The introduction of this slight modification on the sequence $\{p_{n}(x;0)\}_{n\geqslant 0}$ -- more precisely, the inclusion of $x^{-\alpha}$ on the left and of $x^\alpha$ on the right  hand side -- has the merit and pertinency of guaranteeing that we will be able to deal with regular forms as long as $\Re(\alpha)>0$, as it will be analyzed on Section \ref{sec: Dual sequence}. Notwithstanding this advantage, the analysis of $\{p_{n}(x;\alpha)\}_{n\geqslant 0}$ became significantly more hard to deal with. 

Before proceeding into this, we list a few elements of $\{p_{n}(x;\alpha)\}_{n\geqslant 0}$
{\small $$
	\begin{array}{l@{}l@{}l}
 p_1(x;\alpha) & = & x (-2 \alpha -1)+\alpha ^2 \\
 p_2(x;\alpha) & = & x^2 (4 \alpha  (\alpha +2)+3)+x (-2 \alpha  (\alpha  (2
   \alpha +3)+2)-1)+\alpha ^4 \\
 p_3(x;\alpha) & = & -x^3 (2 \alpha +1) (2 \alpha +3) (2 \alpha +5)+x^2 (2 \alpha
   +1) (2 \alpha +3) (3 \alpha  (\alpha +2)+5) \\ & &-x (2 \alpha +1)\left(\alpha ^2+\alpha +1\right) (3 \alpha  (\alpha +1)+1)+\alpha ^6
\end{array}
$$ }
The polynomials $p_{n}(x;\alpha)$ in the variable $x$ are also polynomials in the variable $\alpha$, but of degree $2n$ for each $n\in\mathbb{N}$. This can be deduced from \eqref{pn x alpha +1}. 

\subsection{Connection coefficients with the canonical sequence
}\label{subsec: the cn nu}

On the grounds of \eqref{dif difference eq over small pn}, by performing straightforward computations we deduce that the coefficients $c_{n,\nu}(\alpha)$ defined through  
$$
	p_{n}(x;\alpha) = \sum_{\nu=0}^{n} c_{n,\nu}(\alpha) x^{\nu}
$$
fulfill the recurrence relation 
\begin{equation} \label{con coeff of pn with xn}	
	c_{n+1,\nu} (\alpha)
		= \big( \nu +\alpha \big)^{2} c_{n,\nu}(\alpha) 
		 	- \big(2 \nu+ 2\alpha-1 \big) c_{n,\nu-1}(\alpha) \ , \ 
			0\leqslant \nu\leqslant n, 
  \ , \ n\geqslant 0 , 
\end{equation}
under the convention of $c_{n,-1}(\alpha)=c_{n,n+1}(\alpha)=0$, 
while   \eqref{pn x alpha +1} yields  
\begin{equation} \label{con coeff of pn with xn alpha +1}	
	(2\alpha+1) c_{n,\nu-1}(\alpha+1) = - c_{n+1,\nu}(\alpha) + \alpha^{2} c_{n,\nu}(\alpha) 
	\ , \ 1\leqslant \nu\leqslant n+1 \ , n\geqslant 0. 
\end{equation}
\begin{lemma} The sequence $p_{n}(x;\alpha)$ is then given by 
\begin{equation} \label{pn x alpha explicitly}
	p_{n}(x;\alpha) 
	= \sum_{\nu=0}^n \left\{
	 \frac{2^{\nu+1}(\alpha+1/2)_{\nu}}{\nu!} \sum _{\mu=0}^\nu \binom{\nu}{\mu} \frac{ (-1)^{\mu}   \ 
			\Gamma (2 \alpha +\mu) 
	}{\Gamma (2 \alpha +\mu+\nu+1)}(\alpha +\mu)^{2 n+1}\right\} x^\nu
	\ , \ n\geqslant 0.
\end{equation}
\end{lemma}

\begin{proof}
Let us set $c_{n,\nu}(\alpha)=(-2)^{\nu}(\alpha+1/2)_{\nu} \ \widetilde{c}_{n,\nu}(\alpha)$ in order to deduce an explicit expression for $\widetilde{c}_{n,\nu}(\alpha)$. The ``triangular'' relation \eqref{con coeff of pn with xn} ensures another ``triangular'' relation for the new set of coefficients $\{ \widetilde{c}_{n,\nu}(\alpha)\}_{0\leqslant \nu\leqslant n}$ 
\begin{equation}\label{ctilde Triangular relO2}
\left\{\begin{array}{l}	
	\widetilde{c}_{n+1,\nu} (\alpha)
		=  \widetilde{c}_{n,\nu-1}(\alpha) + 
		 \big( \nu +\alpha \big)^{2} \widetilde{c}_{n,\nu}(\alpha)  \ , \ 
			0\leqslant \nu\leqslant n,   \ , \ n\geqslant 0, \vspace{0.2cm}
\\
  	\widetilde{c}_{n,0} (\alpha)=\alpha^{2n} 
	\quad , \quad 
	\widetilde{c}_{n,\nu} (\alpha)=0 \ , \  \nu\geqslant n+1  \ , \ n\geqslant 0 . 
 \end{array}\right.
\end{equation} 

The particular choice of $n=0,1$ or $2$ in  \eqref{ctilde Triangular relO2} furnishes the identities
\begin{equation*}
	\widetilde{c}_{1,0}(\alpha)=\alpha^2 ,\ \widetilde{c}_{1,1}(\alpha)=1, \ 
	\widetilde{c}_{2,0}(\alpha)=\alpha^4, \ \widetilde{c}_{2,1}(\alpha)=1+2\alpha(\alpha+1), \ \widetilde{c}_{2,2}(\alpha)=1 \ ,
\end{equation*}
which show the validity of the identity 
\begin{align}
	\label{ctilde ExplicitO2}
	& \widetilde{c}_{n,\nu}(\alpha)
		= \frac{2}{\nu!} \sum _{\mu=0}^\nu \binom{\nu}{\mu} \frac{ (-1)^{\mu+\nu}   \ 
			\Gamma (2 \alpha +\mu) 
	}{\Gamma (2 \alpha +\mu+\nu+1)}(\alpha +\mu)^{2 n+1} 
	\ , \ 0\leqslant \nu \leqslant  n
\end{align}
for at least $n=0,1$ and 2. By finite induction, we will show that \eqref{ctilde ExplicitO2} actually holds for any $n\in\mathbb{N}$. Indeed, according to \eqref{ctilde Triangular relO2}, it follows 
$$\begin{array}{l@{\ }l@{\ }l}
	\widetilde{c}_{n+1,\nu}(\alpha) 
		&=&   \ds  \frac{2}{(\nu-1)!} \sum _{\mu=0}^{\nu-1} \binom{\nu-1}{\mu} \frac{ (-1)^{\mu+\nu-1}   \ 
			\Gamma (2 \alpha +\mu) 
	}{\Gamma (2 \alpha +\mu+\nu)}(\alpha +\mu)^{2 n+1} \\
		&& \ds + \frac{2(\nu+\alpha)^2}{\nu!} \sum _{\mu=0}^\nu \binom{\nu}{\mu} \frac{ (-1)^{\mu+\nu}   \ 
			\Gamma (2 \alpha +\mu) 
	}{\Gamma (2 \alpha +\mu+\nu+1)}(\alpha +\mu)^{2 n+1} \\ 
	&=& \ds  \frac{2}{\nu!} \sum _{\mu=0}^\nu \binom{\nu}{\mu} \frac{ (-1)^{\mu+\nu}   \ 
			\Gamma (2 \alpha +\mu) 
	}{\Gamma (2 \alpha +\mu+\nu+1)}(\alpha +\mu)^{2 n+1}
	\Big(-(\nu-\mu)(\nu+\mu+2\alpha) + (\nu+\alpha)^2\Big) \\
	&=& \ds  \frac{2}{\nu!} \sum _{\mu=0}^\nu \binom{\nu}{\mu} \frac{ (-1)^{\mu+\nu}   \ 
			\Gamma (2 \alpha +\mu) 
	}{\Gamma (2 \alpha +\mu+\nu+1)}(\alpha +\mu)^{2 n+3}
\end{array}$$
which corresponds to \eqref{ctilde ExplicitO2} when $n$ is replaced by $n+1$. \end{proof}

The particular choice of $\alpha=0$ gives a simpler expression for $p_{n}(x;0)$ than the one obtained in \cite{Yakubovich2009}. Moreover, in this case, the coefficients $\widetilde{c}_{n,\nu}(0)$ coincide with the {\it $0$-modified Striling numbers of second kind}, a particular case of the $A$-modified Stirling numbers of second kind treated in \cite{Loureiro10} -- they correspond to the {\it Jacobi-Stirling numbers} expounded in \cite{LittleJ} and explored from a purely combinatorial point of view in  \cite{Zeng2009}.
Indeed, the set of numbers $\{\widetilde{c}_{n,\nu}(\alpha)\}$ bear some resemblance with the {\it Jacobi-Stirling numbers}, triggering the problem of giving them some combinatorial significance. To avoid dispersion, we defer this study for a future work.

\subsection{Generating function}\label{subsec GF}
From this point forth we consider  $\Re(\alpha)>0$.  
Using relation (2.16.6.4) of \cite{PrudnikovMarichev} and the reciprocal formula \eqref{KL oper Reciprocal} we obtain the representation 
\begin{equation}\label{integral of ex xalpha}
	 \e^{-x} x^{\alpha}
	= \frac{2^{1-\alpha}}{\pi^{3/2} \Gamma(\alpha+1/2)} \ \int_{0}^{\infty} \tau \sinh(\pi \tau) 
	|\Gamma(\alpha+i\tau)|^{2} K_{i\tau}(x) d\tau. 
\end{equation}
The absolute and uniform convergence of 
\begin{equation}\label{int derivatives K itau}
	\int_{0}^\infty \frac{\partial^m K_{i\tau}(x)}{\partial x^m}  \ \tau^{2n+1} \sinh(\pi \tau) 
	|\Gamma(\alpha+i\tau)|^{2}d\tau\ ,\quad m,n\in\mathbb{N}_{0}
\end{equation}
with respect to $x\geqslant x_{0}>0$, easily verified by taking into account inequality \begin{equation}\label{ineq Kitau}
	\left| \frac{\partial^m K_{i\tau}(x)}{\partial x^m } \right| \leqslant   \e^{-\delta \tau} K_{m}(x\cos \delta) , \quad x>0, \ \tau >0, \ m \in\mathbb{N}_{0}
\end{equation} with $\delta \in (0,\pi/2)$, 
permits to interchange the order between the integral \eqref{integral of ex xalpha} and the operator $\mathcal{A}^n$. Hence, appealing to \eqref{An Kitau} combined with \eqref{integral of ex xalpha}, we derive 
\begin{equation}\label{An ex xalpha}
	\mathcal{A}^{n} \Big( \e^{-x} x^{\alpha}\Big) 
	= \frac{2^{1-\alpha}}{\pi^{3/2} \Gamma(\alpha+1/2)} \ \int_{0}^{\infty} \tau^{2n+1} \sinh(\pi \tau) |\Gamma(\alpha+i\tau)|^{2} K_{i\tau}(x) d\tau  \ , \  n\in\mathbb{N}_{0}. 
\end{equation}
Therefore, recalling  \eqref{def pn nonmonic}, the expression for polynomials $p_{n}(x;\alpha)$, we obtain the integral representation 
\begin{equation}\label{pn ex xalpha Ki tau}
	p_{n}(x;\alpha) = (-1)^{n}\  \frac{2^{1-\alpha} \  \e^{x} x^{-\alpha} }{\pi^{3/2} \Gamma(\alpha+1/2)} \int_{0}^{\infty} \tau^{2n+1} \sinh(\pi \tau) |\Gamma
	(\alpha+i\tau)|^{2} K_{i\tau}(x) d\tau
	\ , \ n\in \mathbb{N}_{0},
\end{equation}
where the latter integral is absolutely convergent for all $x>0$ (see \eqref{int derivatives K itau}).

By setting 
\begin{equation} \label{Generating function F}
	F_{\alpha}(u,x) =\frac{2^{1-\alpha}}{\pi^{3/2} \Gamma(\alpha+1/2)} \ 
	 \e^{x} x^{-\alpha} \int_{0}^{\infty} \cos(\tau u) \ \tau \ \sinh(\pi \tau) \ 
					|\Gamma(\alpha+i\tau)|^{2} K_{i\tau}(x) d\tau
\end{equation}
and by taking into account that $ \frac{\partial^{2n} F_{\alpha}(u,x)}{\partial u^{2n}}, n\in\mathbb{N}_{0},$  is uniformly convergent by $u$ in $\mathbb{R}$ via the absolutely convergent integral \eqref{pn ex xalpha Ki tau}, we obtain 
\begin{equation}\label{Der 2n F alpha and pn}
	\lim_{u\rightarrow 0} \frac{\partial^{2n} F_{\alpha}(u,x)}{\partial u^{2n}} 
	= \frac{2^{1-\alpha}\ (-1)^{n}\ \e^{x} x^{-\alpha}}{\pi^{3/2} \Gamma(\alpha+1/2)} 
	   \int_{0}^{\infty} \tau^{2n+1} \sinh(\pi \tau) 
		|\Gamma(\alpha+i\tau)|^{2} K_{i\tau}(x) d\tau  = p_{n}(x;\alpha) 
					\ , \ n\in\mathbb{N}_{0}, 
\end{equation}
whereas  
\begin{equation}\label{Der 2n+1 F alpha equal 0}
	\lim_{u\rightarrow 0} \frac{\partial^{2n+1} F_{\alpha}(u,x)}{\partial u^{2n+1}} 
	= 0 	\ , \ n\in\mathbb{N}_{0} . 
\end{equation}
As a consequence, the series representation is allowed 
\begin{equation} \label{Generating function F Series Rep}
	F_{\alpha}(t,x) = \sum_{n\geqslant 0} \frac{p_{n}(x;\alpha)}{(2n)!} \ t^{2n} \ .
\end{equation}
We stress that
$$
	\frac{\partial^{2} F_{\alpha}(u,x)}{\partial u^{2}}  
	= - \e^{x}x^{-\alpha} \mathcal{A} \e^{-x}x^{\alpha}  F_{\alpha}(u,x) ,
$$
which, owing to \eqref{A over ex xalpha f}, corresponds to a second order partial differential equation. Besides, the insertion of \eqref{Generating function F Series Rep} into this latter relation yields \eqref{dif difference eq over small pn}. Moreover,
we have 
$$
	\frac{\partial^{2n} F_{\alpha}(u,x)}{\partial u^{2n}}  
	= (-1)^{n} \e^{x}x^{-\alpha} \mathcal{A}^{n} \e^{-x}x^{\alpha}  F_{\alpha}(u,x)\ , \ n\in\mathbb{N}_{0}.
$$

\subsection{Connection with the generalized Euler numbers}\label{subsec Euler gen}

We integrate through \eqref{pn ex xalpha Ki tau} multiplied by $\e^{-2x}x^{\alpha+\epsilon -1}$ with respect to $x$ over $\mathbb{R}^+$, changing the order of integration in its right-hand side due to Fubini's theorem (use inequality \eqref{ineq Kitau}). Now, we invoke  that
$$
	\int_{0}^\infty x^{\epsilon-1}\e^{-x} K_{i\tau}(x)dx = 
		2^{-\epsilon} \sqrt{\pi} \frac{|\Gamma(\epsilon + i\tau)|^2}{\Gamma(\epsilon+1/2)}
$$
to obtain 
$$\begin{array}{l}
	\ds \lim_{\epsilon\to 0}	\int_{0}^\infty \e^{-2x}x^{\alpha+\epsilon -1} p_{n}(x;\alpha) dx\\
	\ds =  \lim_{\epsilon\to 0}\frac{2^{1-\alpha-\epsilon}}{\pi  \Gamma(\alpha+1/2) \Gamma(\epsilon+1/2)} 
	\int_{0}^\infty \tau^{2n+1}\sinh(\pi\tau)|\Gamma(\alpha+i\tau)|^2 {|\Gamma(\epsilon + i\tau)|^2} d\tau \ , \ n\in\mathbb{N}_{0},
\end{array}
$$
which yields  
\begin{equation} \label{integral e 2x pn for Euler gen}
	\int_{0}^\infty \e^{-2x}x^{\alpha -1} p_{n}(x;\alpha) dx
	= \frac{(-1)^n 2^{1-\alpha}}{\sqrt{\pi} \Gamma(\alpha+1/2)} 
	\int_{0}^\infty \tau^{2n}|\Gamma(\alpha+i\tau)|^2 d\tau \ , \ n\in\mathbb{N}_{0}. 
\end{equation}

We recall the {\it Generalized Euler polynomials} $E_{n}^{2\alpha}(x)$ \cite[Vol. III]{Bateman}\cite{Luke,MejriMarGeneralBernoulli,GenEulerPoly} defined through 
\begin{equation}\label{GF Euler gen polys} 
	\left(\frac{2}{{\rm e}^t +1}\right)^{2\alpha} {\rm e}^{xt} 
	= \sum_{n\geqslant 0} E_{n}^{2\alpha}(x) \frac{t^n}{n!}\ .
\end{equation}
The particular choice of $x=\alpha$ brings the identity 
\begin{equation} \label{gen Euler numbers}
	\left(\cosh(t/2)\right)^{-2\alpha} 
	= \sum_{n\geqslant 0} E_{n}^{2\alpha}(\alpha) \frac{t^n}{n!}\ .
\end{equation}
On the other hand, \cite[(1.104)]{YakubovichBook1996} 
$$
	 \left(\cosh(u/2) \right)^{-2\alpha}
	 =  \frac{2^{2\alpha-1}}{\pi\  \Gamma(2\alpha)}
	\int_{0}^{\infty} \cos(\tau u)   |\Gamma(\alpha+i\tau)|^{2} d\tau
$$
providing similarly to \eqref{Der 2n F alpha and pn} the equality
$$
	\lim_{u\rightarrow 0}\frac{\partial^{2n} }{\partial u^{2n}}
	 \left(\cosh(u/2) \right)^{-2\alpha}
	 =   \frac{(-1)^n \  2^{2\alpha-1}}{\pi\  \Gamma(2\alpha)}
	\int_{0}^{\infty}  \tau^{2n}  |\Gamma(\alpha+i\tau)|^{2} d\tau
	\ , \quad n\in\mathbb{N}_{0}.
$$
Hence, \eqref{integral e 2x pn for Euler gen} together with \eqref{gen Euler numbers} imply 
\begin{equation}\label{EulerGenNumbers pn x alpha}
	E_{2n}^{2\alpha}(\alpha) = \frac{2^{\alpha-1}}{ \Gamma(\alpha)} 
	(-1)^{n}  \int_{0}^{\infty}
	 {\rm e}^{-2x} x^{\alpha-1} \ p_{n}(x;\alpha)dx \ , \quad n\in\mathbb{N}_{0},
\end{equation}
which, according to \eqref{def pn nonmonic}, amounts to the same as 
\begin{equation*}\label{EulerGenNumbers pn x alpha 2}
	E_{2n}^{2\alpha}(\alpha) = \frac{2^{\alpha-1}}{ \Gamma(\alpha)} 
	 \int_{0}^{\infty}
	 \frac{{\rm e}^{-x}}{x} \mathcal{A}^n e^{-x} x^{\alpha}dx \ , \quad n\in\mathbb{N}_{0}.
\end{equation*}

Aware of the many formulas that can be found in the literature, we give here a formula that, as far as we are concerned, is new in the theory.

\begin{lemma} For any complex parameter $\alpha$ such that $\Re(\alpha)>0$, the  {\it Generalized Euler polynomials} of parameter $2\alpha$, $E_{n}^{2\alpha}(x)$ are given by 
\begin{equation}\label{Euler polys}
	E_{n}^{2\alpha}(x)
	=  x^{n} 
		+  \sum _{k=0}^{n-1}\left\{ \binom{n}{k} \sum _{\tau =1}^{n-k} \left(-2\right)^{-\tau}
		(2 \alpha )_ \tau\  \mathbf{S}(n-k,\tau)\right\} x^{k}
\end{equation}
where $ \mathbf{S}(k,\tau) $ represent the Stirling numbers of second kind. 
\end{lemma}
\begin{proof} The $n$th order derivative of the left-hand side of \eqref{GF Euler gen polys} evaluated at the point $t=0$ furnishes an expression for 
$E_{n}^{2\alpha}(x)$. The Leibniz rule for the derivative of the product of two functions permits to write  
$$
	\left.\frac{\partial^n}{\partial t^n} \left(\frac{2}{{\rm e}^t +1}\right)^{2\alpha} {\rm e}^{xt} \right|_{t=0} 
	= \sum_{\nu=0}^n \binom{n}{\nu} x^\nu 
	\left.\frac{d^{n-\nu}}{d t^{n-\nu}} \left(\frac{2}{{\rm e}^t +1}\right)^{2\alpha}  \right|_{t=0} 
	\ , \quad n\in\mathbb{N}_{0} \ . 
$$
According to the Fa\`a di Bruno's formula for the $k$th order of derivative of the composition of functions (we refer to  \cite[Chapter III]{Comtet} for notation and a compendium of results for the Bell polynomials and the Fa\`a di Bruno's formula), for any $k\geqslant 1$, we have 
\begin{equation}\label{proof Euler polys Deriv 2}
	\left.\frac{d^{k}}{d t^k} \left(\frac{2}{{\rm e}^t +1}\right)^{2\alpha}  \right|_{t=0} 
	= \left.\sum_{\mu=1}^k (\alpha)_{\mu} (-1)^\mu B_{k,\mu}(\frac{{\rm e}^t}{2},\frac{{\rm e}^t}{2},\frac{{\rm e}^t}{2},\frac{{\rm e}^t}{2},\ldots ) \right|_{t=0} 
\end{equation}
where $B_{n,k}(x_{1},x_{2},\ldots, x_{n-k+1})$, $1\leqslant k \leqslant n$, represents the {\it Bell polynomials} evaluated at the $n$-tuple $(x_{1},x_{2},\ldots, x_{n})$. On the basis of the properties of the Bell polynomials it follows (see \cite[p.135]{Comtet})
$$
	B_{k,\mu}(\frac{{\rm e}^t}{2},\frac{{\rm e}^t}{2},\frac{{\rm e}^t}{2},\frac{{\rm e}^t}{2},\ldots )
	= \left(\frac{{\rm e}^t}{2}\right)^\mu B_{k,\mu}(1,1,1,1,1,\ldots )
	= \left(\frac{{\rm e}^t}{2}\right)^\mu \mathbf{S}(k,\mu)
$$
with $ \mathbf{S}(k,\mu)$ representing the Stirling numbers of second kind \cite[Chapter V]{Comtet}. The insertion of this information into \eqref{proof Euler polys Deriv 2}, ensures 
$$
	\left.\frac{d^{k}}{d t^k} \left(\frac{2}{{\rm e}^t +1}\right)^{2\alpha}  \right|_{t=0} 
	=\sum_{\mu=1}^k (\alpha)_{\mu} (-2)^{-\mu} \mathbf{S}(k,\mu)
$$
whence the result. 
\end{proof}

With $x$ taken equal to $\alpha $ in \eqref{Euler polys}, it follows that $E_{2n+1}^{2\alpha}(\alpha) = 0 $, while 
$$
	E_{2n}^{2\alpha}(\alpha)
	=\alpha ^n + 
		\sum _{k=1}^n \binom{n}{k} \alpha ^{n-k} \sum _{\tau =1}^k \left(-2\right)^{-\tau}
		(2 \alpha )_ \tau \mathbf{S}(k,\tau) \ , \quad n\in\mathbb{N}_{0}, 
$$
are polynomials of degree $n$ in $\alpha$. We list a few examples of the (polynomial) sequence 
$\{E_{2n}^{2\alpha}(\alpha)\}_{n\geqslant0}$:
$$
 \begin{array}{ll}
 E_{2}^{2\alpha}(\alpha) = -\frac{\alpha }{2} &
E_{4}^{2\alpha}(\alpha) = \frac{3 \alpha ^2}{4}+\frac{\alpha }{4} \\
 E_{6}^{2\alpha}(\alpha) = -\frac{15 \alpha ^3}{8}-\frac{15 \alpha ^2}{8}-\frac{\alpha
   }{2} 
   &
E_{8}^{2\alpha}(\alpha) = \frac{105 \alpha ^4}{16}+\frac{105 \alpha ^3}{8}+\frac{147
   \alpha ^2}{16}+\frac{17 \alpha }{8} \\
\end{array}
$$

\subsection{Connection with the Bernoulli polynomials when $\alpha$ is a positive integer}\label{subsec: bernoulli polys}

Consider the (modified) falling factorial of a complex number $x$, 
$$
	[x]_{n}:=\left(\prod_{\sigma=0}^{m-1}(x-\sigma^{2})\right) \ , \ n\in\mathbb{N}, \quad \text{and} \quad 
	[x]_{0}:=1. 
$$
This is actually a polynomial of degree $n$ whose coefficients in the canonical basis are the so-called {\it  $0$-modified Stirling numbers of first kind}, denoted by $\widehat{s}_{0}(n,\nu)$, and treated in \cite{Loureiro10}. Regarding their importance for the subsequent developments, we recall 
\begin{equation}\label{0Stirling1st}
	[x]_{n}=\sum_{\nu=0}^{n} \widehat{s}_{0}(n,\nu) x^{\nu} \ , \quad n\in\mathbb{N}_{0}. 
\end{equation}
This gives grounds for writing $ |\Gamma(\alpha+i\tau)|^{2}$, whenever $\alpha=m\in\mathbb{N}$, as follows: 
$$
	\ds |\Gamma(m+i\tau)|^{2}
	=\ds\left(\prod_{\sigma=0}^{m-1}(\tau^{2}+\sigma^{2})\right) \frac{\pi}{\tau \sinh(\pi \tau)}
	=\frac{ (-1)^{m}\pi}{\tau \sinh(\pi \tau)} [-\tau^2]_{m} 
	= \frac{\pi}{\tau \sinh(\pi \tau)} 
		\sum_{\sigma=0}^{m} (-1)^{m+\sigma}\widehat{s}_{0}(m,\sigma) \tau^{2\sigma} .
$$

As a consequence, the generating function $F_{m}(u,x)$ given by \eqref{Generating function F} becomes  
$$
	\ds F_{m}(u,x) 	= \ds C_{m}  \ \pi \ x^{-m} \e^{x} 
				\sum_{\sigma=0}^{m} (-1)^{m+\sigma}\widehat{s}_{0}(m,\sigma) 
				\int_{0}^{\infty}\tau^{2\sigma} \cos(\tau u) K_{i\tau}(x)\ d\tau 
$$
which amounts to the same as 
\begin{equation}\label{Gen function Fm integer}
	F_{m} (t,x) = C_{m}\ (-1)^{m}\ x^{-m} 
				\sum_{\sigma=0}^{m}  \frac{\pi^{2}}{2} \ 
				\widehat{s}_{0}(m,\sigma) \ 
				\frac{\partial^{2\sigma}}{\partial t^{2\sigma}} \Big(\e^{-2x\sinh^{2}(u/2)}\Big)
\end{equation}
because of the equality \cite[(1.12)]{Yakubovich2009}
$$
	\int_{0}^{\infty}\tau^{2n} \cos(\tau u) K_{i\tau}(x) \ d\tau
	= \frac{(-1)^{\sigma}\ \pi}{2} \ \frac{\partial^{2n}}{\partial t^{2n}} \e^{-x\cosh(u)}\ , \ n\in\mathbb{N}_{0} . 
$$
Concomitantly, the  polynomials $p_{n}(x;m)$ for $m\in\mathbb{N}$ are related to $p_{n}(x;0)$ through 
\begin{equation}\label{Connection pn m with pn 0}
	x^{m} p_{n}(x;m) = C_{m} \ (-1)^{m} \frac{\pi^{2}}{2}  \ \sum_{\sigma=0}^{m} \widehat{s}_{0}(m,\sigma) 
				p_{n+\sigma}(x;0) \  , \ n\geqslant 0 .
\end{equation}
with $\ds \  C_{m} = \frac{2}{\pi}\cdot \frac{1}{2^{m} (1/2)_{m}} $. Recalling \cite[(3.14)]{Yakubovich2009}, it follows a connection with the Bernoulli polynomials evaluated at $\frac{1-i\tau}{2}$
\begin{equation}\label{Connection pn m with Bernoulli}
	x^{m} p_{n}(x;m) = C_{m} \ (-1)^{m} \frac{\pi^{2}}{2}  \ \sum_{\sigma=0}^{m} \widehat{s}_{0}(m,\sigma) 
				\frac{-2^{2(n+1)} \e^{x} }{(2n+1)\pi i} \int_{0}^{\infty}
				\tau \frac{K_{i\tau}(x)}{x} B_{2n+1} \left(\frac{1-i\tau}{2}\right) d\tau
				\  , \ n\geqslant 0 .
\end{equation}

Such equality has the merit of deriving analog properties for  $p_{n}(x;m)$, when $m\in\mathbb{N}$, to those obtained in \cite{Yakubovich2009}. Namely, the connection with the Bernoulli or Euler numbers as well as the Bernoulli polynomials. For further considerations, we refer to \cite[ (3.3),  (3.13)-(3.14) and  (3.17)]{Yakubovich2009}. As a consequence of the aforementioned, we have the following result, which, as far as we are concerned, is new in the theory. 

\begin{lemma} The {\it Generalized Euler numbers} of parameter $2m$ and the {\it Bernoulli numbers} are related by 
$$
	E_{2n}^{2m}(m)= \frac{(-1)^{n+m} 2^{2m-2}}{(2m-1)!} \sum_{\sigma=0}^m \frac{1-2^{2n+2\sigma}}{n+\sigma}
	\widehat{s}_{0}(m,\sigma) B_{2n+2\sigma} \ , \ n\in\mathbb{N}_{0}.
$$
\end{lemma}
\begin{proof} The relation (3.3) pointed out in \cite{Yakubovich2009} furnishes an identity between the Bernoulli numbers and the polynomials $p_{n}(x;0)$ by means of an integral
$$
	B_{2n}=\frac{n}{1-2^{2n}} \int_{0}^\infty {\rm e}^{-2x} p_{n}(x;0) \frac{dx}{x} \ , \ n\in\mathbb{N}_{0}.
$$
The combination of this latter with  \eqref{EulerGenNumbers pn x alpha} together with \eqref{Connection pn m with pn 0} brings to light the desired equality. 
\end{proof}

The forthcoming developments leave aside aspects of the polynomial sequence $\{p_{n}(x;\alpha)\}_{n\geqslant 0}$ itself in order to embrace outer consequences. The study will be essentially focused on the dual sequence corresponding to the monic polynomial sequence, hereafter $\{P_{n}(x;\alpha)\}_{n\geqslant 0}$, obtained from $\{p_{n}(x;\alpha)\}_{n\geqslant 0}$ by dividing each of its elements by the respective leading coefficient: 
\begin{equation} \label{def Pn A n}
	P_{n}(x;\alpha) = \dfrac{1}{(-2)^{n} \ (\alpha +1/2)_{n} } p_{n}(x;\alpha)= \frac{1}{2^{n} \ (\alpha+1/2)_{n}} \ {\rm e}^{x}\ x^{-\alpha} \mathcal{A}^{n}\  {\rm e}^{-x} x^{\alpha} \ , \ 
		\ n\in\mathbb{N}_{0}. 
\end{equation}

\section{\normalfont The monic polynomial sequence and its corresponding dual sequence}\label{sec: Dual sequence}

As, from now on, we will deal exclusively with the monic polynomial sequences $\{P_{n}(x;\alpha)\}_{n\geqslant 0}$, we recreate the needed properties based on those already obtained for the original sequence. Indeed, in accordance with \eqref{dif difference eq over small pn}, $\{P_{n}(x;\alpha)\}_{n\geqslant 0}$ fulfills 
\begin{equation}\label{dif difference eq over Pn}
	 x^2 \ P_n''(x;\alpha) 
	- x\Big(2x - 1 -2\alpha\Big) \ P_n'(x;\alpha)
	- \Big( (2\alpha +1)x - \alpha^{2} \Big) \ P_n(x;\alpha) 
	= - \big( 2n+2\alpha +1 \big) P_{n+1}(x;\alpha)
\end{equation}
for all $n\in\mathbb{N}_{0}$, whereas the equations \eqref{pn x alpha +1}-\eqref{pn x alpha +1 Diff Rel} yield 
\begin{align}
&	 x\ P_{n}(x;\alpha+1) 
		=  P_{n+1}(x;\alpha) + \frac{\alpha^{2}}{(2n+2\alpha +1 )} \ P_{n}(x;\alpha), \quad n\in\mathbb{N}_{0}.
		\label{xPn alpha +1} 
\end{align}

In the next \S\ref{subsec:dual seq} we will ferret out properties of the corresponding dual sequence, $\{u_{n}(\alpha)\}_{n\geqslant 0}$, which will trigger some interesting results, namely, the positive-definite character of the canonical element $u_{0}(\alpha)$ as it will be shown in \S\ref{subsec:regularity}. 

\subsection{The dual sequence}\label{subsec:dual seq}

The differential and difference equations satisfied by the MPS $\{P_{n}(x;\alpha)\}_{n\geqslant 0}$ enable the differential and difference equations fulfilled by the corresponding dual sequence $\{u_{n}(\alpha)\}_{n\geqslant 0}$.

\begin{lemma}\label{eq for dual seq un} The dual sequence $\{u_{n}(\alpha)\}_{n\geqslant 0}$ associated to  $\{P_{n}(x;\alpha)\}_{n\geqslant 0}$ is such that 
\begin{align}
	&\Big( (x^2 u_{0}(\alpha))' + x(2x-(1+2\alpha) ) u_{0}(\alpha)\Big)' - ((1+2\alpha)x-\alpha^2)u_{0}(\alpha)=0 
		\label{dif eq u0}\\
	&
	\ds \Big( (x^2 u_{n+1}(\alpha))' + x(2x-(1+2\alpha) ) u_{n+1}(\alpha)\Big)' 
	-  ((1+2\alpha)x-\alpha^2)u_{n+1}(\alpha) 
	\ds	=(2n+2\alpha +1)u_{n}(\alpha) 	
		\label{dif eq un}
\end{align}
for $n\in\mathbb{N}_{0}$. Moreover the moments of the canonical form $u_{0}$ are given by 
\begin{equation}\label{moments u0}
	\left( u_{0}(\alpha) \right)_{n} = \frac{[(\alpha)_{n}]^2}{2^n (\alpha+1/2)_{n}}\quad , \quad n\in\mathbb{N}_{0}. 
\end{equation}
\end{lemma}

\begin{proof} The action of $u_{0}$ over the relation \eqref{dif difference eq over Pn} furnishes 
$$
	 \langle u_{0}(\alpha) , -x^2 \ P_n''(x) 
	+ x\Big(2x - 1 -2\alpha\Big) \ P_n'(x)
	+ \Big( (2\alpha +1)x - \alpha^{2} \Big) \ P_n(x) \rangle 
	=  0
	 , \quad n\geqslant 0,
$$
which, by transposition, in accordance with \eqref{Ttranspose}, becomes 
$$
	 \left\langle  -\big(x^2 \ u_{0}(\alpha)\big)'' 
	-\big( x\big(2x - 1 -2\alpha\big) \ u_{0} (\alpha)\big)'
	+ \Big( (2\alpha +1)x - \alpha^{2} \Big) u_{0}(\alpha) \ , \ P_n(x) \right\rangle 
	=  0
	 , \quad n\geqslant 0,
$$
whence \eqref{dif eq u0}. Likewise, the action of $u_{k+1}$ over \eqref{dif difference eq over Pn} provides 
$$
	 \ds \left\langle u_{k+1}(\alpha) \ ,\ 
	 -x^2 \ P_n''(x) 
	+ x\Big(2x - 1 -2\alpha\Big) \ P_n'(x)
	+ \Big( (2\alpha +1)x - \alpha^{2} \Big) \ P_n(x) \right\rangle 
	= (2n+2\alpha +1) \ \delta_{n,k}	
	 , \ n\geqslant 0,
$$
and, again, by transposition, this latter leads to 
$$
	 \ds\langle  -\big(x^2 \ u_{k+1}(\alpha)\big)'' 
	-\big( x\big(2x - 1 -2\alpha\big) \ u_{k+1}(\alpha) \big)'
	+ \big( (2\alpha +1)x - \alpha^{2} \big) u_{k+1}(\alpha) \ , \ P_n(x)\rangle 
	\ds = (2n+2\alpha +1) \ \delta_{n,k}	
	 , \ n\geqslant 0,
$$
therefore,  according to \eqref{u in terms of un}, we derive 
\eqref{dif eq un}.
On the other hand, the action of both sides of \eqref{dif eq u0} over each element of the sequence $\{x^n\}_{n\geqslant0}$ leads to the following difference equation having the moments of $u_{0}(\alpha )$ as solution 
\begin{equation*}
	(2n+2\alpha +1) \big( u_{0}(\alpha)\big)_{n+1} =  - (n+\alpha)^{2} \big( u_{0}(\alpha)\big)_{n}
	\quad , \quad \ n\in\mathbb{N}_{0}.
\end{equation*}
providing \eqref{moments u0}. 
\end{proof}

In addition, according to \eqref{u in terms of un},  we have as well
$$
	x u_{0}(\alpha) = \sum_{\nu\geqslant 0} \langle x\  u_{0}(\alpha),P_{n}(x;\alpha+1) \rangle u_{n}(\alpha+1).
$$
Considering \eqref{Ttranspose} and \eqref{xPn alpha +1}, it follows
$$
	\langle x\  u_{0}(\alpha),P_{n}(x;\alpha+1) \rangle 
	= \langle   u_{0}(\alpha),P_{n+1}(x;\alpha) + \frac{\alpha^2}{2n+2\alpha+1}P_{n}(x;\alpha) \rangle 
	= \frac{\alpha^2}{2n+2\alpha+1}\delta_{n,0}, 
$$
that holds for any $n\in\mathbb{N}_{0}$. Therefore the two forms $u_{0}(\alpha)$ and $u_{0}(\alpha+1)$ are related by 
\begin{equation}
	x\ u_{0}(\alpha) = \frac{\alpha^{2}}{2\alpha+1} \ u_{0}(\alpha+1) \ .
	\label{x u0 alpha}
\end{equation}

The question on whether the MPS $\{ P_{n}(x;\alpha)\}_{n\geqslant 0}$ can be orthogonal arises in a natural way and, concomitantly, on whether the form $u_{0}(\alpha)$ is or is not regular. 

As a matter of fact, we recall that a form $v$ is said to be {\it regular}\label{regular form} if we can associate with it a PS $\{Q_{n}\}_{n\geqslant 0}$ such that $\langle v , Q_{n} Q_{m} \rangle = k_{n} \delta_{n,m}$ with $k_{n}\neq0$ for all $n,m\in\mathbb{N}_{0}$ \cite{MaroniTheorieAlg,MaroniVariations}. The PS $\{Q_{n}\}_{n\geqslant 0}$ is then said to be orthogonal with respect to $v$ and we can assume the system (of orthogonal polynomials) to be monic. Therefore, there exists a dual sequence $\{v_{n}\}_{n\geqslant0}$ and the original form is proportional to $v_{0}$. Furthermore, it holds 
\begin{align} &	\label{MOPS un expr} 
	\ v_{n+1} = \left( \langle v_{0}, Q_{n+1}^2(\cdot) \rangle \right)^{-1} \; Q_{n+1}(x) v_{0} 
	\ , \ n\in\mathbb{N}_{0}. 
\end{align}
Moreover, when $v\in\mathcal{P}'$ is regular, let $\Phi$ be a polynomial such that $\Phi v=0$, then $\Phi=0$ \cite{MaroniTheorieAlg,MaroniVariations}.  

In this case we call this unique MPS $\{Q_{n}(x)\}_{n\geqslant 0}$ as {\it monic orthogonal polynomial sequence} - hereafter MOPS - and it can be characterized by the popular second order recurrence relation 
\begin{align} \label{MOPS rec rel} 
&	\left\{ \begin{array}{@{}l}
		Q_{0}(x;\alpha)=1 \quad ; \quad Q_{1}(x;\alpha)= x-\beta_{0}{(\alpha)} \vspace{0.15cm}\\
		Q_{n+2}(x;\alpha) = (x-\beta_{n+1}{(\alpha)})Q_{n+1}(x;\alpha) - \gamma_{n+1}{(\alpha)} \, Q_{n}(x;\alpha)  \quad , \quad n\in\mathbb{N}_{0}. 
	\end{array} \right. 
\end{align}
Here, we consider the dependence on the complex parameter $\alpha$ to avoid repetition of similar formulas. 
We will  systematically refer to the pair $(\beta_{n}{(\alpha)},\gamma_{n+1}{(\alpha)})_{n\geqslant 0}$ as the recurrence coefficients of $\{Q_{n}(x;\alpha)\}_{n\geqslant 0}$, necessarily fulfilling $\gamma_{n+1}{(\alpha)}\neq0,\; n\geqslant0$. See \cite{MaroniTheorieAlg,MaroniVariations} for notation and a compendium of results about algebraic properties of orthogonal polynomial sequences along with regular forms. 

Somehow expected, we have the following result. 

\begin{lemma} The MPS $\{P_{n}(x;\alpha)\}_{n\geqslant 0}$ cannot be regularly orthogonal. 
\end{lemma}

\begin{proof} Under the assumption that $\{P_{n}(x;\alpha)\}_{n\geqslant 0}$ defined through \eqref{def Pn A n} is orthogonal, we may insert relation \eqref{MOPS un expr}, with $v_{n}$ replaced by $u_{n}(\alpha)$ and $Q_{n}$ by $P_{n}(\cdot;\alpha)$, in \eqref{dif eq un} which provides 
\begin{equation}\label{nonortho proof rel1}\begin{array}{@{}r@{}l@{}}
	 \Big( (x^2 P_{n+1}(x;\alpha)u_{0}(\alpha))' + x(2x-(1+2\alpha) )P_{n+1}(x;\alpha)u_{0}(\alpha)\Big)' \\
	-  ((1+2\alpha)x-\alpha^2)P_{n+1}(x;\alpha)u_{0}(\alpha)
		& \ =\lambda_{n}(\alpha)P_{n}(x;\alpha)u_{0}(\alpha)
		, \quad  n\in\mathbb{N}_{0}, 
		\end{array}
\end{equation}
where $\lambda_{n}(\alpha)= (2n+2\alpha +1) <u_{0}(\alpha), P_{n+1}^{2}(x;\alpha) > \left( <u_{0}(\alpha), P_{n}^{2}(x;\alpha)>\right)^{-1}$, for all $n\in\mathbb{N}_{0}$. Taking into account \eqref{dif eq u0}, then \eqref{nonortho proof rel1} implies the differential equation 
$$
	2 P_{n+1}'(x;\alpha) (x^2 u_{0}(\alpha))' 
	+\Big( x^2 P_{n+1}''(x;\alpha)+x(2x-(1+2\alpha) P_{n+1}'(x;\alpha)\Big)u_{0}(\alpha) 
	= \lambda_{n}(\alpha)P_{n}(x;\alpha)u_{0}(\alpha)		, \ n\geqslant 0. 
$$
When the equation obtained by the particular choice of $n=0$, {\it i.e.} 
$$
	2 (x^2 u_{0}(\alpha))' 
	+ x(2x-(1+2\alpha)u_{0}(\alpha) 
	= \lambda_{0}(\alpha)u_{0}(\alpha), 
$$
is inserted into the original equation, we obtain  
$$
	\Big( x^2 P_{n+1}''(x;\alpha)\Big)u_{0}(\alpha) 
	=\Big( \lambda_{n}(\alpha)P_{n}(x;\alpha)-\lambda_{0}(\alpha) P_{n+1}'(x;\alpha)\Big)u_{0}(\alpha). 
$$
The regularity of $u_{0}(\alpha)$ would now imply the condition 
$$ x^2 P_{n+1}''(x;\alpha)= \lambda_{n}(\alpha)P_{n}(x;\alpha)-\lambda_{0}(\alpha) P_{n+1}'(x;\alpha)\ , \ n\geqslant 0, 
$$
contradicting $\deg P_{n} (x;\alpha)=n$, and therefore crumbling the possibility of $\{P_{n}(x;\alpha)\}$ to be orthogonal. 
\end{proof}

Despite the non-(regular)orthogonality of  $\{P_{n}(x;\alpha)\}_{n\geqslant 0}$ with respect to the form $u_{0}(\alpha)$, we cannot exclude the existence of an orthogonal polynomial sequence, say $\{Q_{n}(x;\alpha)\}_{n\geqslant 0}$, with respect to $u_{0}(\alpha)$, which amounts to the same as ensuring the regularity of $u_{0}(\alpha)$. This question is handled in the next section. 

\subsection{About the regularity of  $u_{0}$}\label{subsec:regularity}

We begin by rewriting  \eqref{dif eq u0} as follows 
\begin{equation}\label{dif eq u0 n2}
	\Big( (\phi u_{0}(\alpha))' + \psi u_{0}(\alpha) \Big)' + \chi u_{0}(\alpha) = 0
\end{equation}
with 
\begin{equation}\label{dif eq polys Phi Psi Chi}
	\begin{array}{@{}l@{}}
	\ds \phi(x):=\phi(x;\alpha)= x^2  , \ \psi(x):=\psi(x;\alpha) = x(2x-2\alpha-1) , \ 
	\chi(x):=\chi(x;\alpha)= - (2\alpha+1)x + \alpha^2  .
	\end{array}
\end{equation}

Actually, while seeking an integral representation for  $u_{0}(\alpha)$, we realize that it is indeed regular. 

\begin{lemma} For any positive real value of the parameter $\alpha$, the form $u_{0}(\alpha)$ is positive definite (therefore regular) admitting the integral representation 
\begin{equation}	\label{IntRep of u0}
	\langle u_{0} , f(x) \rangle = \frac{2^{\alpha}  \Gamma(\alpha+1/2)}{\sqrt{\pi}  \ \Gamma(\alpha)^2} 
		\int_{0}^{+\infty} \! f(x) \ \textrm{e}^{-x} x^{\alpha-1} K_{0}(x) \ dx
		\quad , \quad \forall f\in \mathcal{P}. 
\end{equation}
\end{lemma}

\begin{proof} We seek a function $U(x):=U(x;\alpha)$ such that \eqref{IntRep of u0} holds in a certain domain $C$. Since $\langle u_{0} , 1\rangle=1\neq 0$, we must have 
\begin{equation}\label{In U neq 0}
	\int_{C}   U(x) dx =1\neq 0 
\end{equation}
By virtue of \eqref{dif eq u0 n2}, we have, for any $f\in\mathcal{P}$
$$\begin{array}{lll}
	0 	&=& \langle \Big( (\phi(x) u_{0})' + \psi(x) u_{0} \Big)' + \chi(x) u_{0} , f (x)\rangle 
		= \langle u_{0}, \phi(x) f''(x) + \psi(x) f'(x) + \chi(x) f(x)  \rangle \vspace{0.2cm}\\
		&=& \ds \int_{C}\left( (\phi(x) U(x))''+(\psi(x) U(x))'+\chi(x)U(x)\right)f(x) dx \\
		&&	- \left.\Big(\phi(x)U(x)f'(x) - (\phi(x)U(x))'f(x) - \psi(x) U(x)f(x)\Big)\right|_{C} 
\end{array}
$$
therefore, $U(x)$ is a function  simultaneously fulfilling 
\begin{align}
	 & \int_{C}\left( (\phi(x) U(x))''+(\psi(x) U(x))'+\chi(x)U(x)\right)f(x) dx =0\quad ,\quad \forall f\in\mathcal{P}
	 	\label{Cond1 for U}\\
	 & \left.\Big(\phi(x)U(x)f'(x) - (\phi(x)U(x))'f(x) - \psi(x) U(x)f(x)\Big)\right|_{C}=0
	 	\quad ,\quad \forall f\in\mathcal{P} \ .
		\label{Cond2 for U}
\end{align}
The first equation implies 
$$
	 (\phi(x) U(x))''+(\psi(x) U(x))'+\chi(x)U(x) = \lambda g(x)
$$
where $\lambda$ is a complex number and $g(x)\neq 0$ is a function representing the null form, that is, a function such that  
$$
	 \int_{C} g(x) f(x) dx =0 \quad , \quad \forall f\in\mathcal{P}.
$$
We begin by choosing $\lambda=0 $ and we search a regular solution of the differential equation 
$$
	 (\phi(x) U(x))''+(\psi(x) U(x))'+\chi(x)U(x) = 0 
$$
Upon the change of variable $U(x)=\textrm{e}^{-x} x^{\alpha-1} y(x)$ we have 
$$
	x^2 y''(x) + x y'(x) - x^2 y(x) = 0
	\quad \Leftrightarrow \quad
	\mathcal{A} \Big(y(x)\Big) = 0
$$
whose general solution is: \ 
$
	y(x) = c_{1} I_{0}(x) + c_{2} K_{0}(x) \ , \ x\geqslant 0, 
$
for some arbitrary constants $c_{1},c_{2}$ and $y(x)=0$ when $x<0$ \cite{Bateman}. 
As a consequence 
$
	U(x) = \textrm{e}^{-x} x^{\alpha-1} \Big\{c_{1} I_{0}(x) + c_{2} K_{0}(x) \Big\}\ , \ x\geqslant 0. 
$
Insofar as $U(x)$ must be a rapidly decreasing sequence (that is, such that $\lim\limits_{x\to +\infty} f(x)U(x)=0$ for any polynomial $f$), we  set $c_{1}=0$ and $c_{2}\neq 0$ and we write 
$
	U(x) = c_{2} \textrm{e}^{-x} x^{\alpha-1} K_{0}(x) \ , \ x\in C=]0,+\infty[
$
with $c_{2}$ set in order to realize \eqref{In U neq 0}, which amounts to the same as 
$
	c_{2} = \frac{2^{\alpha}  \Gamma(\alpha+1/2)}{\sqrt{\pi}  \ \Gamma(\alpha)^2}\ .
$
In this case $U(x)$ also fulfills the condition \eqref{In U neq 0} as long as $\alpha > 0 $ because 
$
	\int_0^{+\infty} \textrm{e}^{-x} x^{\alpha-1} K_{0}(x) dx  
$ is a stricktly positive convergent integral. 
Besides, since $K_{\nu}(x)$ has the asymptotic behaviour with respect to $x$ \cite[Vol. II]{Bateman}\cite{YakubovichBook1996}
\begin{align*}
	& K_{\nu}(x) = \left(\frac{\pi}{2x}\right)^{1/2}\e^{-x}[1+O(1/x)] , \quad x\rightarrow +\infty
		\label{Knu at infty}\\
	& K_{\nu}(x) =O(x^{-\Re(\nu)}) \ , K_{0}(x) =O(-\log x)  \ , \quad x\rightarrow 0, 
\end{align*}
then we immediately deduce that \eqref{Cond2 for U}  \eqref{Cond2 for U} is fulfilled by any element of the PS $\{x^n\}_{n\geqslant 0}$ (that spans $\mathcal{P}$) and hence it is fulfilled by any element $f$ of $\mathcal{P}$.  
Moreover, for every polynomial $g(x)$ that is not identically zero and is non-negative for all real $x$ we have 
$\ 
\ds	\langle u_{0} ,  g(x) \rangle =  \frac{2^{\alpha}  \Gamma(\alpha+1/2)}{\sqrt{\pi}  \ \Gamma(\alpha)^2}\int_{0}^{+\infty} \! g(x) \ \textrm{e}^{-x} x^{\alpha-1} K_{0}(x) \ dx > 0 \ 
$ 
and therefore $u_0$ is {\bf positive-definite} form. This implies that $u_0$ has real moments and a corresponding MOPS exists ({\it i.e.}, $u_0$ is a regular form). 
\end{proof}
The regularity of $u_0(\alpha)$ raises the problem of characterizing a corresponding orthogonal polynomial sequence, say $\{Q_n(x;\alpha)\}_{n\geqslant 0}$, whenever $\alpha$ is a positive real number. From this point forth we consider $ \alpha\in\mathbb{R}^+$.

Entailed in the question of characterizing the arisen MOPS $\{Q_{n}(x;\alpha)\}_{n\geqslant 0}$ (with respect to $u_{0}(\alpha)$), comes out the problem of determining the associated recurrence coefficients  or even a differential equation fulfilled by $\{Q_{n}(x;\alpha)\}_{n\geqslant 0}$ or an analogue of Rodrigues' formula or a generating function. However, this has revealed to be a tricky problem to solve, reminding other open problems such as the one posed by Prudnikov \cite{PrudnikovProblem}. Indeed, the difficult part of this problem is connected to the fact that the regular form $u_{0}(\alpha)$ is solution of a differential equation of order higher than one. As far as we can tell the gap in the theory concerned to this problem, avoids us to attain further results. 
Notwithstanding this, based on the work  \cite{MaroniRelTypeFini}, it is possible to reach a finite-type relation (of order two actually) between the two MOPSs $\{Q_{n}(x;\alpha)\}_{n\geqslant 0}$ and $\{Q_{n}(x;\alpha+1)\}_{n\geqslant 0}$. Again, this adds very few to the answer of the problem. 

The recurrence coefficients $(\beta_{n}{(\alpha)}, \gamma_{n+1}{(\alpha)})_{n\geqslant 0}$ associated to the second order recurrence relation \eqref{MOPS rec rel} fulfilled by $\{Q_n(\cdot;\alpha)\}_{n\geqslant 0}$ may be successively computed either by means of the weight function or using the Hankel determinant of the moments of $u_{0}(\alpha)$ \cite{ChiharaBook}. Making use of the first process, for each positive integer $n$ we have 
\begin{align}
&	\beta_n{(\alpha)} = \frac{\langle u_0(\alpha), x Q_n^2(x;\alpha) \rangle}
		{\langle u_0(\alpha), Q_n^2 (x;\alpha)\rangle} 
		= \frac{\ds\int_{0}^{+\infty} \!  Q_{n}^2(x;\alpha) \ \textrm{e}^{-x} x^{\alpha} K_{0}(x) \ dx}
		{\ds\int_{0}^{+\infty} \!  Q_{n}^2(x;\alpha) \ \textrm{e}^{-x} x^{\alpha-1} K_{0}(x) \ dx}\\
&	\gamma_{n+1}{(\alpha)} =  \frac{\langle u_0(\alpha),  Q_{n+1}^2 (x;\alpha) \rangle}
			{\langle u_0(\alpha), Q_n^2(x;\alpha) \rangle}
			=  \frac{\ds\int_{0}^{+\infty} \!  Q_{n+1}^2(x;\alpha) \ \textrm{e}^{-x} x^{\alpha-1} K_{0}(x) \ dx}
		{\ds\int_{0}^{+\infty} \!  Q_{n}^2(x;\alpha) \ \textrm{e}^{-x} x^{\alpha-1} K_{0}(x) \ dx}.
\end{align}
According to this latter, we list the first elements 
$$\begin{array}{lll}
 \beta_{0}{(\alpha)}= \frac{\alpha ^2}{2 \alpha +1} &\quad , \quad &
  \beta_{1}{(\alpha)}= \frac{\alpha  (2 \alpha  (\alpha +4)+7) (\alpha  (2 \alpha +5)+4)+4}{(2
   \alpha +1) (2 \alpha +5) (2 \alpha  (\alpha +2)+1)} \\
    \gamma_1{(\alpha)} =\frac{\alpha ^2 (2 \alpha  (\alpha +2)+1)}{(2 \alpha +1)^2 (2 \alpha
   +3)}  &\quad , \quad &
\gamma_2{(\alpha)} = \frac{4 (\alpha +1)^2 (2 \alpha +1) (\alpha  (\alpha  (2 \alpha 
   (\alpha  (2 \alpha  (\alpha +12)+113)+262)+613)+325)+51)}{(2 \alpha
   +3) (2 \alpha +5)^2 (2 \alpha +7) (2 \alpha  (\alpha +2)+1)^2} \ .
\end{array}
$$
The determination of recurrence coefficients of higher order has indeed revealed to be a ticklish problem.


\begin{thebibliography}{amsplain}
\bibitem{Bateman} A. Erd\'elyi, W. Magnus, F. Oberhettinger and F. G. Tricomi, Higher Transcendental Functions, Vols. I, II and III, McGraw-Hill, New York, London, Toronto, 1953. 
\bibitem{ChiharaBook} T.S. Chihara, An Introduction to Orthogonal
Polynomials, Gordon and Breach, New York, 1978.
\bibitem{Comtet} L. Comtet, Advanced Combinatorics: The Art of Finite and Infinite Expansions. rev. enl. ed. Dordrecht, Netherlands: Reidel, 1974.
\bibitem{LittleJ} W. N. Everitt, K. H. Kwon, L. L. Littlejohn, R. Wellman, G. J. Yoon, Jacobi-Stirling numbers,
Jacobi polynomials, and the left-definite analysis of the classical Jacobi differential expression,
J.Combut.Appl.Math., 208 (2007), 29-56.
\bibitem{Zeng2009} Y. Gelineau, J. Zeng, Combinatorial interpretations of the Jacobi-Stirling numbers, 15 pages  (2009). Retrieved from \url{http://arxiv.org/abs/0905.2899}
\bibitem{Lebedev} N.N. Lebedev, Sur une formule d'inversion, C.R. Acad. Sci. URSS 52 (1946) pp. 655-658. 
\bibitem{Loureiro10} A.F. Loureiro, New results on the Bochner condition about classical orthogonal polynomials, J. Math. Anal. Appl. 364 (2010) 307-323. 
\bibitem{Luke} Yu. L. Luke, Mathematical Functions and Their Approximations,Academic Press, New York, San Francisco, London, 1975. 
\bibitem{MaroniTheorieAlg}P. Maroni, Une th\'eorie alg\'ebrique des polyn\^{o}mes
orthogonaux. Application aux polyn\^{o}mes orthogonaux semi-classiques,
C. Brezinski et al. (Eds.), Orthogonal Polynomials and their Applications, in:
\textit{IMACS Ann. Comput. Appl. Math.} {9} (1991), 95-130. 
\bibitem{MaroniVariations}P. Maroni, Variations around classical orthogonal polynomials. Connected problems. \textit{Journal of Comput. Appl. Math.}, {48} (1993), 133-155. 
\bibitem{MaroniRelTypeFini} P. Maroni, Semi-classical character and finite-type relations between polynomial sequences, J. Appl. Num. Math. 31 (1999) 295-330.
\bibitem{MejriMarGeneralBernoulli}P. Maroni, M. Mejri, Generalized Bernoulli polynomials revisited and some other Appell sequences. Georgian Math. J. 12(4) (2005) pp. 697–716.
\bibitem{PrudnikovMarichev} A. P. Prudnikov, Yu. A. Brychkov and O.I. Marichev, Integrals and Series. Vol. I: Elementary Functions, Vol. II: Special Functions, Gordon and Breach, New York, London, 1986. 
\bibitem{PrudnikovProblem} A.P. Prudnikov, Orthogonal polynomials with ultra-exponential weight functions, in: W. Van Assche (Ed.), Open Problems, J. Comput. Appl. Math. 48 (1993) 239.
\bibitem{GenEulerPoly} H.M. Srivastava, A. Pinter,
Remarks on some relationships between the Bernoulli and Euler polynomials, Appl. Math. Lett., 17(4) (2004), pp. 375-380.
\bibitem{Yakubovich2009} S. Yakubovich, A class of polynomials and discrete transformations associated with the Kontorovich-Lebedev operators,  Integral Transforms Spec. Funct. 20 (2009), 551-567. 
\bibitem{Yakubovich2004JAT}Yakubovich, S.B., On the least values of Lp-norms for the Kontorovich-Lebedev transform and its convolution. Journal of Approximation Theory, 131, 231-242 (2004)
\bibitem{Yakubovich2003} S. B. Yakubovich, On the Kontorovich-Lebedev transformation, J. Integral Equations Appl. 15(1) (2003) pp. 95-112. 
\bibitem{YakubovichBook1996} S. B. Yakubovich, Index Transforms, World Scientific Publishing Company, Singapore, New Jersey, London and Hong Kong, 1996.
\bibitem{YakubovichBook1994} S.B. Yakubovich, B. Fisher, On the theory of the Kontorovich-Lebedev transformation on distributions, Proc. of the Amer. Math. Soc., 122 (1994), 773-777.
\end{thebibliography}
\end{document}